\documentclass[a4paper]{amsart}

\usepackage[utf8]{inputenc}
\usepackage{hyperref}
\usepackage{amssymb}
\usepackage{amsthm}
\usepackage{amsmath}
\usepackage{mathtools}
\usepackage{enumitem}
\usepackage{nicefrac}
\usepackage{xcolor}

\usepackage{url}

\begin{document}
	
\theoremstyle{plain}
\newtheorem{lemma}{Lemma}
\newtheorem{proposition}[lemma]{Proposition}
\newtheorem{corollary}[lemma]{Corollary}
\newtheorem{theorem}[lemma]{Theorem}
\newtheorem*{theorem*}{Theorem}
	
\theoremstyle{definition}
\newtheorem{definition}[lemma]{Definition}
\newtheorem{example}[lemma]{Example}

\theoremstyle{remark}
\newtheorem{remark}[lemma]{Remark}

\newcommand{\period}{\text{.}}
\newcommand{\comma}{\text{,}}

\newcommand{\rca}{\textsf{RCA}}
\newcommand{\wkl}{\textsf{WKL}}
\newcommand{\aca}{\textsf{ACA}}
\newcommand{\atr}{\textsf{ATR}}
\newcommand{\pica}{\Pi^1_1\textsf{-CA}_0}
\newcommand{\ads}{\textsf{ADS}}
\newcommand{\n}{\mathbb{N}}
\newcommand{\nplus}{\n \setminus \{0\}}
\newcommand{\set}{\textsf{Set}}
\newcommand{\po}{\textsf{PO}}
\newcommand{\wpo}{\textsf{WPO}}
\newcommand{\supp}{\textsf{supp}}
\newcommand{\rng}{\textsf{rng}}
\newcommand{\fin}{\textsf{fin}}

\newcommand{\red}[1]{\textcolor{red}{#1}}
\newcommand{\REF}{\red{REF}}
	
\title{The uniform Kruskal theorem over $\rca_0$}
\author[P. Uftring]{Patrick Uftring}

\address{Patrick Uftring, University of W\"urzburg, Institute of Mathematics, Emil-Fischer-Stra{\ss}e~40, 97074 W\"urzburg, Germany}
\email{patrick.uftring@uni-wuerzburg.de}
\thanks{This work has been funded by the Deutsche Forschungsgemeinschaft (DFG, German Research Foundation) -- Project number 460597863.}
\keywords{Dilators, uniform Kruskal theorem, arithmetical comprehension, chain antichain principle, ascending descending sequence principle, infinite pigeonhole principle, exponentiation, perfect sequence, well order, well partial order}

\begin{abstract}
	Kruskal's theorem famously states that finite trees (ordered using an infima-preserving embeddability relation) form a \emph{well partial order}. Freund, Rathjen, and Weiermann extended this result to general recursive data types with their \emph{uniform Kruskal theorem}. They do not only show that this principle is true but also, in the context of \emph{reverse mathematics}, that their theorem is equivalent to~${\Pi^1_1}$-comprehension, the characterizing axiom of~${\pica}$.
	
	However, their proof is not carried out directly over~${\rca_0}$, the usual base system of reverse mathematics. Instead, it additionally requires a weak consequence of Ramsey's theorem for pairs and two colors: the chain antichain principle.
	
	In this article, we show that this additional assumption is not necessary and the considered equivalence between the uniform Kruskal theorem and $\Pi^1_1$-comprehension already holds over~${\rca_0}$. For this, we improve Girard's characterization of arithmetical comprehension using ordinal exponentiation by showing that his result even remains correct if only a certain subclass of well orders is considered.
\end{abstract}

\maketitle

\section{Introduction}

Recursive data types are one of the most fundamental concepts in computer science and formal mathematics. Using a certain set of constructors (e.g., borrowing notation from type theory,~${0: X}$ and~${\texttt{succ}: X \to X}$), the resulting data type is given by all possible closed terms that they allow us to form (here, the resulting set~${X}$ corresponds to the natural numbers~${\n}$).

Order relations on these constructors can induce partial orders on the ensuing data types:
E.g., sequences~${L^*}$ with members from a partial order~${L}$ result from two constructors~${\texttt{empty}: X}$ and~${\texttt{cons}: L \to X \to X}$. Now, for any partial order~${X}$, we set~${\texttt{empty} \leq \texttt{empty}}$ and~${\texttt{cons}\ l\ x \leq \texttt{cons}\ k\ y}$ for~${l, k \in L}$ and~${x, y \in X}$ satisfying~${l \leq k}$ and~${x \leq y}$. This map from a partial order~${X}$ to the partial order of our constructors applied to~${X}$ will be called a \emph{\po-dilator} (in Definition~\ref{def:po_dilator}).
In order to define a partial order on our data type~${L^*}$, we consider the definition above (recursively) for~${X := L^*}$ and add a further relation~${x \leq \texttt{cons}\ l\ y}$ for~${l \in L}$ and~${x, y \in L^*}$ with~${x \leq y}$. The resulting order on~${L^*}$ is precisely the one employed in Higman's lemma (cf.~\cite{Higman}), i.e.~if~${L}$ is a \emph{well partial order} (see Definition~\ref{def:wpo}), then the same must hold for~${L^*}$.

As a further example, binary trees~${T}$ can be constructed using~${\texttt{empty}: X}$ as well as~${\texttt{pair}: X \to X \to X}$. Given a partial order~${X}$, our \po-dilator sets~${\texttt{empty} \leq \texttt{empty}}$ and~${\texttt{pair}\ x\ y \leq \texttt{pair}\ u\ v}$ for~${x, y, u, v \in X}$ with~${x \leq u}$ and~${y \leq v}$. We define an order on~${X := T}$ by applying these definitions and~${x \leq \texttt{pair}\ u\ v}$ for~${x, u, v \in T}$ with~${x \leq u}$ or~${x \leq v}$, recursively. The resulting order is given by infima-preserving embeddings as they are employed in Kruskal's theorem for binary trees (cf.~\cite[Theorem~1]{Kruskal}), i.e.~$T$ is a well partial order.

In reverse mathematics, we measure the strength of statements in second-order arithmetic over a weak base theory~${\rca_0}$ with restricted induction and comprehension axioms (see \cite{SimpsonBook}). By not only proving theorems from axioms but also axioms from theorems (hence, the name \emph{reverse} mathematics), we precisely determine the axioms that are required to prove a given theorem. E.g., while the result that~${L \mapsto L^*}$ preserves partial orders is quite weak and can already be derived in~${\rca_0}$, the statement that it preserves \emph{well} partial orders (i.e.~Higman's lemma) is not only provable using the axiom of \emph{arithmetical comprehension} but also, over~${\rca_0}$, this statement allows us to derive arithmetical comprehension itself (cf.~\cite[Theorem~3]{Clote}, which combines \cite[Section~II.5]{GirardExp} with \cite[Lemma~5.2]{SchuetteSimpson}, \cite[Sublemma~4.8]{SimpsonHilbert}). A similar analysis can be done for the order on binary trees from earlier (even though this particular example does not depend on any other partial order as input).

While many theorems from ordinary mathematics are equivalent to one of the \emph{Big Five}, i.e.~$\rca_0$,~${\wkl_0}$,~${\aca_0}$ (characterized by arithmetical comprehension),~${\atr_0}$, or~${\pica}$, there are also some outliers. E.g., already the statement that our second example of binary trees yields a well partial order does not match any of these systems precisely: It cannot be derived in~${\aca_0}$ but it is also not strong enough to imply the characterizing axiom of the next system~${\atr_0}$ (cf.~\cite{Schmidt75, Schmidt77}, also \cite[Section~6]{FreundReification}). Similarly, allowing arbitrary finite trees or introducing (as little as two) labels yields a principle that is not provable in~${\atr_0}$ (cf.~\cite[Theorem~2.8]{SimpsonNichtbeweisbarkeit, SimpsonNonprovability} and \cite[Theorem~14.11]{SvMW}). Even combining both modifications and even considering arbitrary well partial orders $L$ as labels, we stay strictly below~${\pica}$.\footnote{This version of Kruskal's theorem can be proved using Nash-Williams' minimal bad sequence argument (cf.~\cite[Theorem~1]{NashWilliams}), which is available in $\pica$ (cf.~\cite[Theorem~6.5]{Marcone}), but by its quantifier complexity, this statement must stay below $\pica$ (cf.~\cite[Theorem~VII.2.10]{SimpsonBook}, see \cite[Theorem~4.1.1]{UftringPhD} for more details).}

Previously, we used \po-dilators to extend orders on constructors to partial orders on their least fixed points. The \emph{uniform Kruskal theorem} due to Freund, Rathjen, and Weiermann now reveals that for the large subclass of \emph{normal \wpo-dilators} (see Definitions~\ref{def:po_dilator} and~\ref{def:normal}), the ensuing recursive data type is always a well partial order. Since there exist normal \wpo-dilators that let us construct the trees involved in Kruskal's theorem (i.e.~finitely branching trees with labels), the uniform Kruskal theorem can be seen as a generalization of this result. Freund, Rathjen, and Weiermann do not only show that the uniform Kruskal theorem is true but, in the context of reverse mathematics, they even prove the following characterization:

\begin{theorem*}[{\cite[Theorem~5.12]{FRWUKT}}]
	The system~${\rca_0}$ together with the \emph{chain antichain principle} proves an equivalence between
	\begin{enumerate}[label=\alph*)]
		\item~${\Pi^1_1}$-comprehension, the characterizing axiom of~${\pica}$, and
		\item the uniform Kruskal theorem.
	\end{enumerate}
\end{theorem*}

Notice that this theorem not only assumes the axioms of our usual base system~${\rca_0}$ but also the chain antichain principle (cf.~\cite[p.~178]{HirschfeldtShore} and \cite[Question~13.8]{CJS}), a weak consequence of Ramsey's theorem for pairs and two colors. In \cite[Remark~2.7]{FUKruskal}, Freund and the author could show that this equivalence still holds if we replace the chain antichain principle with the \emph{ascending descending sequence principle} (cf.~\cite[p.~178]{HirschfeldtShore}), which is strictly weaker (cf.~\cite{LST}). Moreover, in the form of two conservativity results, they proved that such an additional assumption is, in fact, necessary to give weaker variants of the uniform Kruskal theorem strength (cf.~\cite[Theorems~3.4 and~3.9]{FUKruskal}).

In this article, we show that the equivalence between~${\Pi^1_1}$-comprehension and the uniform Kruskal theorem does \emph{not} require any further assumptions, i.e., we prove the following:

\begin{theorem}\label{thm:UKT_RCA}
	The system~${\rca_0}$ proves an equivalence between
	\begin{enumerate}[label=\alph*)]
		\item~${\Pi^1_1}$-comprehension and
		\item the uniform Kruskal theorem.
	\end{enumerate}
\end{theorem}

This reveals that the uniform Kruskal theorem has, over~${\rca_0}$, precisely the same strength as the closely related Bachmann-Howard fixed point theorem due to Freund (cf.~\cite{FreundPhD, FreundCategorical, FreundCAWellordering, FreundComputable}), which acts on linear orders instead of partial ones, and has already seen many fruitful applications (cf.~\cite{FreundFromKruskalToFriedman, FreundPredicative, FreundPatterns, FreundDerivative} including \cite{FRWUKT}, the work that this article is based on).

It should be emphasized that Theorem~\ref{thm:UKT_RCA} is not a minor technical extension of the original equivalence. In light of the already mentioned conservativity results involving a weakened uniform Kruskal theorem, it could very well have been the case that the full uniform Kruskal theorem \emph{also} does not posses any strength in the absence of the ascending descending sequence principle. Notice that without this principle, it is a priori not possible to apply this theorem to \po-dilators containing constructors of arity $2$ or higher.
See also \cite[Theorem~4.2 and Corollary~4.5]{FMPS} for a further example of principles that experience significant proof-theoretic strength (here, $\Pi^1_2$-comprehension) over a certain system (here, $\atr_0$) but fail to show any comparable strength over weaker base systems (here, $\aca_0$).
Finally, quoting personal communication with Anton Freund, ``proving the equivalence over $\rca_0$ alone has been considered an important goal, but previous attempts did not work out''.

In order to derive Theorem~\ref{thm:UKT_RCA}, we consider the following classical result due to Girard, which is used by Freund, Rathjen, and Weiermann to bootstrap~${\aca_0}$ using the uniform Kruskal theorem:
\begin{theorem*}[{\cite[Section~II.5]{GirardExp} and \cite[Theorem~2.6]{HirstExp}}]
	The system~${\rca_0}$ proves the following to be equivalent:
	\begin{enumerate}[label=\alph*)]
		\item arithmetical comprehension and
		\item the map~${\alpha \mapsto 2^\alpha}$ preserves well orders.
	\end{enumerate}
\end{theorem*}
Using Kruskal fixed points of a certain class of normal \po-dilators, one can show~b), which leads to arithmetical comprehension by this theorem. However, an application of the uniform Kruskal theorem, which is used to obtain fixed points that are well partial orders, requires the employed \po-dilators to already be \emph{\wpo-dilators}. For this, the chain antichain principle is used. While the ascending descending sequence principle suffices, as mentioned earlier, it can be seen that $\rca_0$ on its own is not strong enough to show that this class only consists of \wpo-dilators. Therefore, in this article, we consider and prove the following variation:
\begin{theorem}\label{thm:exp_perfect}
	The system~${\rca_0}$ proves the following to be equivalent:
	\begin{enumerate}[label=\alph*)]
		\item arithmetical comprehension and
		\item the map~${\alpha \mapsto 2^\alpha}$ preserves well orders~${\alpha}$ for which each infinite suborder contains a strictly ascending sequence.
	\end{enumerate}
\end{theorem}
Now, this restricts the class of \po-dilators that we have to consider for deriving~b). In fact, already~${\rca_0}$ itself can be used to prove that any element of this restriction is a \wpo-dilator. In conclusion, we arrive at Theorem~\ref{thm:UKT_RCA}.

Our construction proceeds in three parts: In Section~\ref{sec:ipp}, we show that the uniform Kruskal theorem yields the infinite pigeonhole principle. In Section~\ref{sec:exp}, we derive Theorem~\ref{thm:exp_perfect}. In Section~\ref{sec:boot}, we combine these results in order to prove Theorem~\ref{thm:UKT_RCA}. Additionally, in Section~\ref{sec:wkl}, we use the main result of Section~\ref{sec:ipp} to show that there is no \po-dilator that, over~${\rca_0}$, becomes a \wpo-dilator if and only if weak K\H{o}nig's lemma holds.
The results of this article correspond to Chapter~5 of the author's PhD thesis (cf.~\cite{UftringPhD}).

\subsection*{Acknowledgments}

I would like to thank my PhD advisor Anton Freund and I am grateful for the funding by the Deutsche Forschungsgemeinschaft (DFG, German Research Foundation) -- Project number 460597863.

\section{Preliminaries}

Well partial orders generalize the concept of well orders from linear orders to the class of partial orders:
\begin{definition}\label{def:wpo}
	Let~${X}$ be a partial order. A sequence~${(x_n)_{n \in \n} \subseteq X}$ is called \emph{good} if there are indices~${n < m}$ with~${x_n \leq x_m}$. Otherwise, it is called \emph{bad}. We say that~${X}$ is a \emph{well partial order} if and only if any such sequence is good.
	
	A function~${f: X \to Y}$ between two partial orders is a \emph{quasi embedding} if it \emph{reflects} the order, i.e.~$f(x) \leq f(x')$ implies~${x \leq x'}$ for any two elements~${x, x' \in X}$. It is an \emph{embedding} if it additionally \emph{preserves} the order, i.e.~$x \leq x'$ implies~${f(x) \leq f(x')}$ for any two elements~${x, x' \in X}$.
	
	We write \po{} and \wpo{} for the categories of partial orders and well partials orders, respectively (with quasi embeddings as morphisms).
\end{definition}
Sometimes, we like to compare finite subsets of a partial order as follows:
\begin{definition}
	Let~${X}$ be a partial order. For any two finite subsets~${F, G \subseteq X}$, we write~${F \leq_\fin G}$ if and only if for any element~${x \in F}$, there is some~${y \in G}$ with~${x \leq y}$.
	In case of a singleton set~${F = \{x\}}$, we may also write~${x \leq_\fin G}$.
\end{definition}
Multiple partial orders can be combined to new ones using the following operators:
\begin{definition}
	Let~${X}$ and~${Y}$ be partial orders.
	
	We define the \emph{sum}~${X + Y}$ to be the set consisting of pairs~${(0, x)}$ and~${(1, y)}$ for~${x \in X}$ and~${y \in Y}$. We order the elements of~${X + Y}$ by setting~${(0, x) \leq (0, x')}$ if and only if~${x \leq x'}$ holds and~${(1, y) \leq (1, y')}$ if and only if~${y \leq y'}$ holds.
	
	We define the \emph{product}~${X \times Y}$ to be the set consisting of pairs~${(x, y)}$ for~${x \in X}$ and~${y \in Y}$. For any two such pairs~${(x, y), (x', y') \in X \times Y}$, we have~${(x, y) \leq (x', y')}$ if and only if both~${x \leq x'}$ and~${y \leq y'}$ hold.
\end{definition}
Later, in the proof of Lemma~\ref{lem:ipp_unary}, we will refer to the elements~${(0, (0, x))}$,~${(0, (1, y))}$, and~${(1, z)}$ for~${x \in X}$,~${y \in Y}$, and~${z \in Z}$ of a sum~${X + Y + Z}$ (assuming left associativity) by~${(0, x)}$,~${(1, y)}$,~${(2, z)}$ in order to ease notation.

In~${\rca_0}$, it is easy to see that~${X + Y}$ and~${X \times Y}$ both are partial orders if the same already holds for~${X}$ and~${Y}$. Moreover, the system~${\rca_0}$ proves that the sum~${X + Y}$ is a well partial order if we can assume the same for~${X}$ and~${Y}$. This can easily be seen using the infinite pigeonhole principle for two colors, which is available in~${\rca_0}$. While a similar result is \emph{true} for products, the statement that~${X \times Y}$ is a well partial order if the same holds for~${X}$ and~${Y}$ is no longer provable in~${\rca_0}$. In fact, by a result due to Towsner (cf.~\cite{Towsner}), its strength lies strictly between that of the ascending descending sequence principle and the chain antichain principle.

\begin{definition}
	The functor~${[\cdot]^{<\omega}: \set \to \set}$ is maps each set~${S}$ to the collection~${[S]^{<\omega}}$ of all finite subsets of~${S}$.
	Given a map~${f: S \to T}$ between two sets, we have~${[f]^{<\omega}(F) := f(F)}$ for each finite subset~${F \subseteq S}$.
\end{definition}

\begin{definition}[{\cite[Definition~2.2]{FRWUKT}}]\label{def:po_dilator}
	A functor~${W: \po \to \po}$ together with a natural transformation\footnote{We are omitting the forgetful functor~${F: \po \to \set}$. To be precise, our natural transformation has the signature~${\supp: F \circ W \Rightarrow [\cdot]^{<\omega} \circ F}$.}~${\supp: W \Rightarrow [\cdot]^{<\omega}}$ is a \emph{\po-dilator} if and only if
	\begin{enumerate}[label=\roman*)]
		\item it preserves embeddings and
		\item the following \emph{support condition}
		\begin{equation*}
			\rng(W(f)) = \{\sigma \in W(Y) \mid \supp_Y(\sigma) \subseteq \rng(f)\}
		\end{equation*}
		is satisfied for any embedding~${f: X \to Y}$.
	\end{enumerate}
	We call~${W}$ a \wpo-dilator if it preserves well partial orders.
\end{definition}
Since \po-dilators preserve embeddings, the support condition allows them to be defined by their behavior on finite partial orders. This lets us code all \po-dilators that are required for our purposes using sets in the context of second-order arithmetic (cf.~\cite[Theorem~2.11]{FRWUKT}) similar to predilators on linear orders (cf.~\cite[Corollary~2.1.8]{Girard}). We leave this coding opaque and refer to (\cite[Definitions~2.2 and~2.4]{FRWUKT}) where an additional distinction is made between \emph{class-sized} \po-dilators (i.e.~Definition~\ref{def:po_dilator}) and \emph{coded} \po-dilators (their actual representation as sets in second-order arithmetic).

The definition of Kruskal fixed points is restricted to a certain class of \po-dilators:
\begin{definition}\label{def:normal}
	A \po-dilator~${W: \po \to \po}$ with support functions~${\supp}$ is \emph{normal} if and only if for any partial order~${X}$ and any two elements~${\sigma, \tau \in W(X)}$ with~${\sigma \leq \tau}$, we have~${\supp_X(\sigma) \leq_\fin \supp_X(\tau)}$.
\end{definition}

Most \po-dilators that we are interested in satisfy this property of normality. See \cite[p.~5]{FRWUKT} for a discussion. In particular, it seems to be required for proving that the construction of the canonical Kruskal fixed point yields a partial order (cf.~\cite[Proposition~3.6]{FRWUKT}).

\begin{definition}[{\cite[Definition~3.7]{FRWUKT}}]
	Consider a normal \po-dilator~${W}$. A \emph{Kruskal fixed point} of~${W}$ is given by a partial order~${X}$ with a bijection~${\kappa: W(X) \to X}$ such that~${\kappa(\sigma) \leq \kappa(\tau)}$ holds if and only if one of
	\begin{enumerate}[label=\roman*)]
		\item~${\sigma \leq \tau}$, or
		\item~${\kappa(\sigma) \leq_\fin \supp_X(\tau)}$
	\end{enumerate}
	does, for any two~${\sigma, \tau \in W(X)}$.
\end{definition}

Now, everything is prepared for stating the principle whose strength over~${\rca_0}$ we aim to determine precisely in this article:

\begin{theorem}[Uniform Kruskal Theorem]
	Any normal \wpo-dilator has a Kruskal fixed point that is a well partial order.
\end{theorem}
In \cite[Theorem~5.12]{FRWUKT}, the uniform Kruskal theorem is defined by a different formulation using a computable transformation~${\mathcal{T}}$ that maps normal (coded) \po-dilators to canonical notation systems of their (initial) Kruskal fixed points. However, in this article, we like to avoid the required coding and, therefore, define the uniform Kruskal theorem as statement \cite[Theorem~5.12~(3)]{FRWUKT}. In light of \cite[Theorem~3.8]{FRWUKT} both formulations are equivalent over~${\rca_0}$.

Let us repeat the characterization of this principle in the context of reverse mathematics due to Freund, Rathjen, and Weiermann:

\begin{theorem}[{\cite[Theorem~5.12]{FRWUKT}}]\label{thm:UKT_CAC}
	The system~${\rca_0}$ together with the chain antichain principle proves an equivalence between
	\begin{enumerate}[label=\alph*)]
		\item~${\Pi^1_1}$-comprehension and
		\item the uniform Kruskal theorem.
	\end{enumerate}
\end{theorem}
As already mentioned in the introduction, it is the main purpose of this article to show Theorem~\ref{thm:UKT_RCA}, i.e.~that the equivalence of Theorem~\ref{thm:UKT_CAC} already holds over~${\rca_0}$ even in the absence of the chain antichain principle.

\section{Deriving the infinite pigeonhole principle}\label{sec:ipp}

The first step in this direction is given by the following result:

\begin{proposition}\label{prop:ukt_ipp}
	The system~${\rca_0}$ proves that the uniform Kruskal theorem implies the infinite pigeonhole principle.
\end{proposition}
Recall, the infinite pigeonhole principle states that for any number~${n \in \n}$ and function~${f: \n \to n}$, there is an infinite (so-called \emph{homogeneous}) set~${X \subseteq \n}$ such that~${f}$ restricted to~${X}$ is constant. In other words: There is a color~${c < n}$ such that~${f^{-1}(c)}$ is infinite.
The heart of the proof lies in the fact that the infinite pigeonhole principle can be characterized using the following well (partial) ordering principle:
\begin{lemma}\label{lem:ipp_alpha_times_n}
	The system~${\rca_0}$ proves that the infinite pigeonhole principle is equivalent to the statement that~${\alpha \times n}$ is a well partial order for each well order~${\alpha}$ and number~${n \in \n}$.
\end{lemma}
\begin{proof}
	First, assume that the infinite pigeonhole principle holds. For contradiction, let~${((\beta_i, k_i))_{i \in \n} \subseteq (\alpha \times n)}$ be a sequence that we assume to be bad in~${\alpha \times n}$. By the infinite pigeonhole principle, there is a number~${k \in \n}$ such that~${k = k_i}$ holds for infinitely many indices~${i \in \n}$. Using~${\Delta^0_1}$-comprehension, we can move to a subsequence and assume, w.l.o.g., that~${k = k_i}$ holds for \emph{all} indices~${i \in \n}$. Then,~${(\beta_i, k) \nleq (\beta_j, k)}$ implies~${\beta_i \nleq \beta_j}$ for any two indices~${i < j}$. Thus,~${(\beta_i)_{i \in \n}}$ is a bad sequence in~${\alpha}$, which is a contradiction.
	
	For the other direction, let~${n \in \n}$ be a number and~${f: \n \to n}$ an arbitrary function such that~${f}$ maps each natural number to a color in~${n}$. Assume, for contradiction, that there does not exist an infinite homogeneous set for~${f}$ while~${\alpha \times n}$ is a well partial order for each well order~${\alpha}$.
	We define~${\alpha}$ to be the following order on~${\n}$: We have~${i <_\alpha j}$ if and only if~${f(i) > f(j)}$ holds or we have both~${f(i) = f(j)}$ and~${i >_\n j}$. It is easy to see that~${\alpha}$ is a linear order. Assume that~${\alpha}$ contains an infinitely descending sequence~${(\beta_i)_{i \in \n} \subseteq \alpha}$. By definition of our order, we know that~${(f(\beta_i))_{i \in \n} \subseteq n}$ must be weakly ascending. Since there are no infinitely ascending sequences in~${n}$, there must be an index~${i \in \n}$ such that~${(f(\beta_j))_{j \geq i}}$ is constant. By definition of our order, the members of this sequence are strictly increasing in~${\n}$. This defines an infinite subset~${X \subseteq \n}$ such that~${f}$ is constant on~${X}$. Hence,~${X}$ is an infinite homogeneous set, which leads to a contradiction. Now, by assumption, we know that~${\alpha \times n}$ must be a well partial order. However, we claim that the sequence~${(i, f(i))_{i \in \n} \subseteq \alpha \times n}$ is bad in~${\alpha \times n}$: Let~${i, j \in \n}$ be arbitrary indices with~${i <_\n j}$. If~${f(i) < f(j)}$ holds, we have~${i >_\alpha j}$, which leads to~${(i, f(i)) \nleq (j, f(j))}$. We arrive at the same conclusion even more directly if~${f(i) > f(j)}$ holds. Now, in the last case with~${f(i) = f(j)}$, our assumption~${i <_\n j}$ leads to~${i >_\alpha j}$. Again, we have~${(i, f(i)) \nleq (j, f(j))}$. We conclude that~${(i, f(i))_{i \in \n}}$ is bad in~${\alpha \times n}$. From this contradiction, we derive that~${f}$ must have an infinite homogeneous set.
\end{proof}
Interestingly enough, it is easy to see that this equivalence does not hold anymore if we only consider well orders~${\alpha}$ for which each infinite suborder contains a strictly ascending sequence. In fact, for this class of well orders, our claim already holds in~${\rca_0}$ itself. This can be used to derive the infinite pigeonhole principle from~${\ads}$.

Our next steps are dedicated to deriving that~${\alpha \times n}$ is a well partial order for arbitrary well orders~${\alpha}$ and numbers~${n \in \n}$ if the uniform Kruskal theorem is available. For this, of course, we require a \wpo-dilator.
\begin{definition}
	For each well order~${\alpha}$, we define a functor~${V_\alpha}$ with
	\begin{equation*}
		V_\alpha(X) := \alpha + X
	\end{equation*}
	for each partial order~${X}$. Given a quasi embedding~${f: X \to Y}$, we set
	\begin{equation*}
		V_\alpha(f)(\sigma) :=
		\begin{cases}
			(0, \beta) & \text{if~${\sigma = (0, \beta)}$,}\\
			(1, f(x)) & \text{if~${\sigma = (1, x)}$,}
		\end{cases}
	\end{equation*}
	for each~${\sigma \in V_\alpha(X)}$. Finally, for each partial order~${X}$ and element~${\sigma \in V_\alpha(X)}$, the support of~${\sigma}$ is given by
	\begin{equation*}
		\supp_X(\sigma) :=
		\begin{cases}
			\emptyset & \text{if~${\sigma = (0, \beta)}$,}\\
			\{x\} & \text{if~${\sigma = (1, x)}$.}
		\end{cases}
	\end{equation*}
\end{definition}
\begin{lemma}\label{lem:W_alpha_normal_wpo_dilator}
	The system~${\rca_0}$ proves that~${V_\alpha}$ is a normal \wpo-dilator for each well order~${\alpha}$.
\end{lemma}
\begin{proof}
	It is easy to see that~${V_\alpha}$ satisfies the functor axioms and preserves embeddings. Also, it is not hard to show that~${\supp}$ is a natural transformation. Concerning the support condition: Let~${f: X \to Y}$ be an embedding and let~${\sigma \in V_\alpha(Y)}$ be arbitrary. If it is of the form~${\sigma = (0, \beta)}$ for some~${\beta \in \alpha}$, then we already have~${\sigma \in V_\alpha(X)}$ and conclude~${V_\alpha(f)(\sigma) = \sigma}$. Otherwise, if~${\sigma = (1, y)}$ holds for some~${y \in Y}$ with~${\{y\} = \supp_Y(\sigma) \subseteq \rng(f)}$, then there exists~${x \in X}$ with~${f(x) = y}$. Consequently, we have~${V_\alpha(f)((1, x)) = (1, f(x)) = (1, y)}$. We conclude that~${V_\alpha}$ is a \po-dilator.
	
	In order to show that~${V_\alpha}$ even is a \wpo-dilator, we consider~${V_\alpha(X)}$ for a well partial order~${X}$. By the infinite pigeonhole principle for two colors, which is available in~${\rca_0}$, any bad sequence in~${V_\alpha(X)}$ leads either to an infinite descending sequence in~${\alpha}$ or a bad sequence in~${X}$. Each, of course, contradicts one of our assumptions. We conclude that~${V_\alpha}$ is a \wpo-dilator.
	
	Finally, we show that~${V_\alpha}$ is normal. Let~${X}$ be an arbitrary partial order and let~${\sigma, \tau \in V_\alpha(X)}$ be elements satisfying~${\sigma \leq \tau}$. Then, we know that either both~${\sigma}$ and~${\tau}$ live in the left summand of~${\alpha + X}$ (i.e.~their first component equals~${0}$), or both of them live in its right summand (i.e.~their first component equals~${1}$). In the first case, their supports are empty and we are immediately done. Otherwise, there exist~${x, y \in X}$ with~${\sigma = (1, x)}$ and~${\tau = (1, y)}$, which entails~${x \leq y}$ by our assumption~${\sigma \leq \tau}$. We conclude
	\begin{equation*}
		\supp_X(\sigma) = \{x\} \leq_\fin \{y\} = \supp_X(\tau)\period
	\end{equation*}
	Hence,~${V_\alpha}$ is a normal \wpo-dilator.
\end{proof}
\begin{lemma}\label{lem:ukt_alpha_times_n}
	The system~${\rca_0}$ proves that~${\alpha \times n}$ quasi-embeds into any uniform Kruskal fixed point of~${V_\alpha}$, for each well order~${\alpha}$ and number~${n \in \n}$.
\end{lemma}
\begin{proof}
	Let~${X}$ be a uniform Kruskal fixed point of~${V_\alpha}$ with collapse~${\kappa: V_\alpha(X) \to X}$. We define a quasi embedding~${f: \alpha \times n \to X}$ recursively by setting
	\begin{equation*}
		f((\beta, i)) :=
		\begin{cases}
			\kappa((0, \beta)) & \text{if~${i = 0}$,}\\
			\kappa((1, f((\beta, i - 1)))) & \text{if~${i > 0}$,}
		\end{cases}
	\end{equation*}
	for each~${(\beta, i) \in \alpha \times n}$. By induction along~${i + j}$, we prove that~${f((\beta, i)) \leq_X f((\gamma, j))}$ implies~${(\beta, i) \leq_{\alpha \times n} (\gamma, j)}$ for all pairs~${(\beta, i), (\gamma, j) \in \alpha \times n}$.
	If~${i + j}$ is equal to zero, then we have~${\kappa((0, \beta)) \leq_X \kappa((0, \gamma))}$. Since the support of~${(0, \gamma)}$ is empty in~${V_\alpha(X)}$, this implies~${(0, \beta) \leq_{\alpha + X} (0, \gamma)}$ and, therefore,~${\beta \leq_\alpha \gamma}$. Finally, we arrive at our claim~${(\beta, 0) \leq_{\alpha \times n} (\gamma, 0)}$.
	
	If~${i}$ is positive but~${j}$ equals zero, then we arrive at a contradiction: The inequality~${f((\beta, i)) \leq_X f((\gamma, 0))}$ is defined as~${\kappa((1, f((\beta, i - 1)))) \leq_X \kappa((0, \gamma))}$. Since the support of~${(0, \gamma)}$ is empty in~${V_\alpha(X)}$, this implies~${(1, f((\beta, i - 1))) \leq_{\alpha + X} (0, \gamma)}$, which is impossible as both elements are incomparable.
	
	If~${i}$ equals zero but~${j}$ is positive, then~${f((\beta, 0)) \leq_X f((\gamma, j))}$ leads to the inequality~${\kappa((0, \beta)) \leq_X \kappa((1, f((\gamma, j - 1))))}$. As both~${(0, \beta)}$ and~${(1, f((\gamma, j - 1)))}$ are incomparable, we conclude that~${\kappa((0, \beta))}$ must be less than or equal to the only element in the support of~${(1, f((\gamma, j - 1)))}$, which is given by~${f((\gamma, j - 1))}$. Applying the induction hypothesis to~${f((\beta, 0)) = \kappa((0, \beta)) \leq_X f((\gamma, j - 1))}$, we arrive at~${(\beta, 0) \leq_{\alpha \times n} (\gamma, j - 1)}$, which clearly leads to our claim~${(\beta, 0) \leq_{\alpha \times n} (\gamma, j)}$.
	
	Finally, if both~${i}$ and~${j}$ are positive, then~${f((\beta, i)) \leq_X f((\gamma, j))}$ implies the inequality~${\kappa((1, f((\beta, i - 1)))) \leq_X \kappa((1, f((\gamma, j - 1))))}$. Now, if this holds because of~${(1, f((\beta, i - 1))) \leq_X (1, f((\gamma, j - 1)))}$, we derive~${f((\beta, i - 1)) \leq_X f((\gamma, j - 1))}$. By induction hypothesis, we have~${(\beta, i - 1) \leq_{\alpha \times n} (\gamma, j - 1)}$. Clearly, this leads to our claim. Otherwise, if the inequality holds since the object on the left hand side is less than or equal to the only support element of the object on the right hand side, we have~${f((\beta, i)) = \kappa((1, f((\beta, i - 1)))) \leq_X f((\gamma, j - 1))}$. By induction hypothesis, this entails~${(\beta, i) \leq_{\alpha \times n} (\gamma, j - 1)}$. Clearly, this leads to our claim~${(\beta, i) \leq_{\alpha \times n} (\gamma, j)}$.
\end{proof}

\begin{proof}[Proof of Proposition~\ref{prop:ukt_ipp}]
	Assume that the uniform Kruskal theorem holds. In order to produce the infinite pigeonhole principle using Lemma~\ref{lem:ipp_alpha_times_n}, we only have to show that~${\alpha \times n}$ is a well partial order for any well order~${\alpha}$ and number~${n \in \n}$. By the uniform Kruskal theorem, there exists a well partial order~${X}$ that together with a collapse~${\kappa: V_\alpha(X) \to X}$ forms a Kruskal fixed point of~${V_\alpha}$. For this, we used Lemma~\ref{lem:W_alpha_normal_wpo_dilator}, according to which~${V_\alpha}$ is a normal \wpo-dilator. Finally, by Lemma~\ref{lem:ukt_alpha_times_n}, there exists a quasi embedding from~${\alpha \times n}$ into~${X}$. Since quasi embeddings reflect the property of being a well partial order, the requirements for Lemma~\ref{lem:ipp_alpha_times_n} are satisfied and we conclude that the infinite pigeonhole principle must hold.
\end{proof}

\section{Exponentiation for a restricted class of well partial orders}\label{sec:exp}

In this section, we derive Theorem~\ref{thm:exp_perfect}, an improvement of Girard's characterization (cf.~\cite[Section~II.5]{GirardExp}) of arithmetical comprehension using exponentiation that will be crucial for our proof of Theorem~\ref{thm:UKT_RCA}. We closely follow Hirst's construction (cf.~\cite[Theorem~2.6]{HirstExp}).

\begin{definition}[{ess.~\cite[Definition~2.1]{HirstExp}}]
	For each well order~${\alpha}$, we define an order~${2^\alpha}$ that consists of terms
	\begin{equation*}
		\sigma := 2^{\beta_0} + \dots + 2^{\beta_{n-1}}
	\end{equation*}
	for each~${n \in \n}$ and elements~${\beta_0 > \dots > \beta_{n-1}}$ of~${\alpha}$.
	Given two such terms~${\sigma}$ and
	\begin{equation*}
		\tau := 2^{\gamma_0} + \dots + 2^{\gamma_{m-1}}
	\end{equation*}
	in~${2^\alpha}$, we have~${\sigma < \tau}$ if and only if
	\begin{enumerate}[label=\roman*)]
		\item~${\tau}$ is a proper extension of~${\sigma}$, i.e.~we have~${n < m}$ as well as~${\beta_i = \gamma_i}$ for all~${i < n}$, or
		\item there is a common index~${i < \min(n, m)}$ satisfying both~${\beta_i < \gamma_i}$ and~${\beta_j = \gamma_j}$ for all~${j < i}$.
	\end{enumerate}
\end{definition}
It can easily be seen in~${\rca_0}$ that~${2^\alpha}$ is a linear order for each well order~${\alpha}$.\footnote{In fact, this property also holds if the underlying linear order is ill-founded. An interesting such instance is given by~${2^{-\omega}}$, where~${-\omega}$ describes the non-positive whole numbers: The resulting order is isomorphic to the rationals (see, e.g.,~\cite[Theorem~3.1]{FreundManca}).}
We closely follow Hirst's construction in \cite[Theorem~2.6]{HirstExp}:
\begin{definition}[{ess.~\cite[p.~4]{HirstExp}}]
	For any injective function~${f: \nplus \to \n}$, we define a tree~${T_f \subseteq \n^*}$ that contains a sequence~${s \in \n^*}$ if and only if \emph{both} of the following are satisfied:
	\begin{enumerate}[label=\roman*)]
		\item The equality~${f(s_i) = i}$ holds for all indies~${i < |s|}$ with~${s_i > 0}$.
		\item The inequality~${s_i > 0}$ holds if there is some index~${j < |s|}$ with~${f(j) = i}$, for all indices~${i < |s|}$.
	\end{enumerate}
	This tree~${T_f}$ is ordered using the \emph{Kleene-Brouwer order} (cf.~\cite[Definition~V.1.2]{SimpsonBook}), i.e., for any two sequences~${s, t \in T_f}$, we have~${s < t}$ if and only if one of the following two is satisfied:
	\begin{enumerate}[label=\roman*), start=3]
		\item The sequence~${s}$ is a proper extension of~${t}$, i.e., we have both~${|s| > |t|}$ and~${s_i = t_i}$ for all~${i < |t|}$.
		\item There is a common index~${i < \min(|s|, |t|)}$ with both~${s_i < t_i}$ and~${s_j = t_j}$ for all smaller indices~${j < i}$.
	\end{enumerate}
\end{definition}
Notice that in contrast to our order on~${2^\alpha}$ for well orders~${\alpha}$, longer sequences represent smaller elements in the Kleene-Brouwer order. It is not hard to find a function~${f}$ so that~${T_f}$ is ill-founded. Still, we have the following:
\begin{lemma}\label{lem:T_f_linear_order}
	The system~${\rca_0}$ proves that for any injective~${f: \nplus \to \n}$, the tree~${T_f}$ (associated with the Kleene-Brouwer order) is a linear order.
\end{lemma}
Before we continue with the main arguments, we introduce a weaker notion of extension for sequences:
\begin{definition}
	Given two sequences~${s, t \in \n^*}$, we call~${t}$ a \emph{weak extension} of~${s}$ if and only if we have~${|s| \leq |t|}$ and~${s_i = t_i}$ for all~${i < |s|}$ with~${s_i > 0}$.
\end{definition}
For general sequences, this notion of a weak extension is almost meaningless. However, in the context of~${T_f}$ for injective~${f}$, it plays an important role as the next result reveals:
\begin{lemma}\label{lem:pos_equal}
	The system~${\rca_0}$ proves that for any injective~${f: \nplus \to \n}$ and sequences~${s, t \in T_f}$ if both~${s_i > 0}$ and~${t_i > 0}$ hold, then we have~${s_i = t_i}$, for each common index~${i < \min(|s|, |t|)}$.
\end{lemma}
\begin{proof}
	This follows immediately from condition i) together with the injectivity of~${f}$.
\end{proof}
With this lemma, we see that~${T_f}$ behaves like a binary tree for injective~${f}$. In fact, it admits even more structure: E.g., it is easy two see that any two sequences in~${T_f}$ have a common weak extension if~${f}$ is injective.

The next lemma contains the main arguments required for proving Theorem~\ref{thm:exp_perfect}. It constructs the claimed strictly ascending sequences:
\begin{lemma}\label{lem:T_f_perfect_sequence}
	The system~${\rca_0}$ proves that for any injective~${f: \nplus \to \n}$, either the range of~${f}$ exists as a set or any infinite suborder of~${T_f}$ contains a sequence~${(s^n)_{n \in \n}}$ satisfying the following:
	\begin{enumerate}[label=\alph*)]
		\item The lengths of its members are strictly increasing.
		\item Each member is a weak extension of all previous ones.
		\item The inequality~${s^n < s^{n+1}}$ holds for infinitely many indices~${n \in \n}$.
		\item The sequence is strictly ascending.
	\end{enumerate}
\end{lemma}
While we are mostly interested in d) for proving Theorem~\ref{thm:exp_perfect}, items a)-c) reveal the steps we take to derive our claim. Clearly, c) is subsumed by d) and, as we will see, implied by a) and b).
\begin{proof}[Proof of Lemma~\ref{lem:T_f_perfect_sequence}]
	Assume that the range of~${f}$ does not exist as a set. Let~${(s^n)_{n \in \n}}$ enumerate the elements of a given infinite subset of~${T_f}$. In particular, its members are pairwise distinct.
	
	For a), we show that for each index~${n \in \n}$, there exists~${m > n}$ such that~${|s^n| < |s^m|}$ holds. Assume, for contradiction, that~${n \in \n}$ is a counterexample to this claim. Let~${g: T_f \to 2^*}$ be the function that maps each~${s \in T_f}$ to a sequence~${g(s)}$ of the same length such that~${s_i = 0}$ holds if and only if~${g(s)_i = 0}$ does, for all indices~${i < |s| = |g(s)|}$. By Lemma~\ref{lem:pos_equal}, this function~${g}$ is injective. It is easy to see that~${2^*}$ contains exactly~${2^{|s_n| + 1} - 1}$ sequences of length less than or equal to~${|s_n|}$. Hence, we conclude that there must be indices~${i, j \in \n}$ with~${n \leq i < j < n + 2^{|s_n| + 1}}$ satisfying~${g(s_i) = g(s_j)}$. Since~${g}$ is injective, we deduce~${s^i = s^j}$. With~${i < j}$, this contradicts our initial assumption that all members in~${(s^n)_{n \in \n}}$ are pairwise distinct. We conclude that for each index~${n \in \n}$, there is some~${m > n}$ with~${|s^n| < |s^m|}$. By moving to a subsequence, we can, w.l.o.g., assume that~${|s^n| < |s^{n+1}|}$ holds for all~${n \in \n}$.
	
	For b), we define a (primitive recursive) function~${g: \n \to \n}$ with
	\begin{equation*}
		g(n) :=
		\begin{cases}
			0 & \text{if~${n = 0}$,}\\
			1 + \max(\{g(n - 1)\} \cup \{(s^{g(n-1)})_i \mid i < |s^{g(n-1)}|\}) & \text{otherwise,}
		\end{cases}
	\end{equation*}
	for all~${n \in \n}$. Immediately, we see that~${g}$ is strictly increasing. Moreover, we claim that~${s^{g(n+1)}}$ is a weak extension of~${s^{g(n)}}$ for all~${n \in \n}$. Consider an arbitrary such index. First, we have~${|s^{g(n)}| < |s^{g(n+1)}|}$ since both~${g}$ and the lengths of the members in~${(s^n)_{n \in \n}}$ are strictly increasing. Now, if~${(s^{g(n)})_i}$ is positive for some index~${i < |s^{g(n)}|}$, then~${f((s^{g(n)})_i) = i}$ holds. By definition of~${g(n+1)}$, we have~${(s^{g(n)})_i < g(n+1)}$. Now, by the fact that the lengths of our members in~${(s^n)_{n \in \n}}$ are strictly increasing, we have~${g(n+1) \leq |s^{g(n+1)}|}$. In conclusion, the inequality~${(s^{g(n)})_i < |s^{g(n+1)}|}$ together with the fact~${f((s^{g(n)})_i) = i}$ and condition ii) lead to~${(s^{g(n+1)})_i > 0}$. By Lemma~\ref{lem:pos_equal} and the injectivity of~${f}$, we arrive at~${(s^{g(n)})_i = (s^{g(n+1)})_i}$. Hence,~${s^{g(n+1)}}$ is a weak extension of~${s^{g(n)}}$.
	By moving to the subsequence induced by~${g}$, we can assume, w.l.o.g., that~${s^{n+1}}$ is a weak extension of~${s^n}$. It is easy to see that the relation of being a weak extension is transitive.
	
	As we will see, c) now automatically follows from a) and b). Assume, for contradiction, that there is some index~${n \in \n}$ starting from which~${s^m \geq s^{m+1}}$ holds for all~${m \geq n}$. By a), we know that each of these inequalities must be strict. Given~${m \geq n}$, assume that~${s^m > s^{m+1}}$ holds since there is a common index~${i < \min(|s^m|, |s^{m+1}|)}$ with both~${(s^m)_i > (s^{m+1})_i}$ and~${(s^m)_j = (s^{m+1})_j}$ for all~${j < i}$. Already the first property contradicts our assumption that~${s^{m+1}}$ is a weak extension of~${s^m}$. Thus, by definition of the Kleene-Brouwer order,~${s^{m+1}}$ must be a proper extension of~${s^m}$, for all~${m \geq n}$. We claim that this yields the range of~${f}$ as a set. To be precise, let~${X}$ be the set that contains a number~${i \in \n}$ if and only if~${(s^{\max(n,i) + 1})_i > 0}$ holds. By a), it is clear that we have~${i < |s^{\max(n,i) + 1}|}$, which guarantees the expression ``$(s^{\max(n,i) + 1})_i > 0$'' to be well-defined. We claim that this set~${X}$, which can be constructed using~${\Delta^0_1}$-comprehension, is the range of~${f}$. First, consider~${i \in X}$. This implies~${(s^{\max(n,i) + 1})_i > 0}$. By definition of~${T_f}$, we have~${f((s^{\max(n,i) + 1})_i) = i}$. For the other direction, assume that~${i \in \n}$ lies in the range of~${f}$ witnessed by~${j \in \nplus}$, i.e.~we have~${f(j) = i}$. From a), we know~${|s^{\max(n, i, j) + 1}| > i, j}$. Hence, condition ii) implies~${(s^{\max(n, i, j) + 1})_i > 0}$. Now, since~${s^{\max(n, i, j) + 1}}$ is an extension of~${s^{\max(n, i) + 1}}$ by assumption, we conclude that the inequality~${(s^{\max(n, i) + 1})_i > 0}$ holds and, therefore,~${i \in X}$. However, the existence of~${X}$ contradicts our assumption that the range of~${f}$ cannot be given by a set. Finally, this contradiction leads to~${s^m < s^{m+1}}$ for some~${m \geq n}$.
	
	For d), let~${h: \n \to \n}$ be a strictly increasing function satisfying~${s^{h(n)} < s^{h(n) + 1}}$ for all~${n \in \n}$. Such a function can easily be extracted from a witness for c).
	Let~${n, m \in \n}$ be arbitrary indices with~${n < m}$. We claim that~${s^{h(n)} < s^{h(m)}}$ holds. For this, consider the inequality~${s^{h(n)} < s^{h(n) + 1}}$. Since, by a),~${s^{h(n) + 1}}$ is strictly longer than~${s^{h(n)}}$, there must be a common index~${i < \min(|s^{h(n)}|, |s^{h(n) + 1}|)}$ satisfying~${(s^{h(n)})_i < (s^{h(n) + 1})_i}$ and~${(s^{h(n)})_j = (s^{h(n) + 1})_j}$ for all~${j < i}$. By Lemma~\ref{lem:pos_equal} and the injectivity of~${f}$, the former entails~${(s^{h(n)})_i = 0}$. Of course, we also have the inequality~${0 < (s^{h(n) + 1})_i}$. Now, from~${h(n) + 1 \leq h(m)}$ and b), we derive that~${s^{h(m)}}$ is a weak extension of~${s^{h(n) + 1}}$. In particular, this implies~${0 < (s^{h(m)})_i }$. Now, because of~${(s^{h(n)})_i = 0}$, the sequence~${s^{h(m)}}$ cannot be an extension of~${s^{h(n)}}$. Assume, for contradiction, that~${s^{h(n)} \geq s^{h(m)}}$ holds. Combining our thoughts from above with the definition of the Kleene-Brouwer order, we conclude that there must be a common index~${i' < \min(|s^{h(n)}|, |s^{h(m)}|)}$ with~${(s^{h(m)})_{i'} < (s^{h(n)})_{i'}}$ and~${(s^{h(m)})_{j'} = (s^{h(n)})_{j'}}$ for all~${j' < i'}$. However, the former contradicts b), according to which~${s^{h(m)}}$ must be a weak extension of~${s^{h(n)}}$. From this contradiction, we conclude~${s^{h(n)} < s^{h(m)}}$. Finally, by moving to the subsequence induced by~${h}$, we arrive at our claim.
\end{proof}
The only missing piece for deriving the negation of Theorem~\ref{thm:exp_perfect}~b) is given by the next result:
\begin{lemma}\label{lem:2_T_f_ill_founded}
	The system~${\rca_0}$ proves that for any injective~${f: \nplus \to \n}$, the order~${2^{T_f}}$ is ill-founded.
\end{lemma}
For this result, we could simply refer to the proof of \cite[Theorem~2.6]{HirstExp}. However, in the presence of Lemma~\ref{lem:T_f_perfect_sequence}, there exists a simpler argument:
\begin{proof}[Proof of Lemma~\ref{lem:2_T_f_ill_founded}]
	Although it is not important for our theorem, let us consider the case where the range of~${f}$ exists as a set~${X}$. Then, it is easy to construct an ill-founded sequence~${(s^n)_{n \in \n}}$ in~${T_f}$ by defining~${s^n}$, for each~${n \in \n}$, to be the sequence of length~${n}$ with~${f((s^n)_i) = i}$ if~${i \in X}$ and~${(s^n)_i = 0}$ if~${i \notin X}$, for each~${i < n}$. Using~${X}$, we can decide whether we are in the former or latter case, for each index. Moreover, if we are in the former case, a simple search, that is guaranteed to terminate by definition of~${X}$, lets us find a witness~${(s^n)_i}$, for each such~${i < n}$. Clearly, this ill-founded sequence in~${T_f}$ can be translated to one in~${2^{T_f}}$.
	
	We continue under the assumption that the range of~${f}$ does \emph{not} exist as a set.	
	We begin by defining a family of partial maps~${g_p:\subseteq T_f \to 2^{T_f}}$. For each sequence~${p * s \in T_f}$, we set	
	\begin{equation*}
		g_p(s) :=
		\begin{cases}
			0 & \text{if~${s = \langle \rangle}$,}\\
			2^p + g_{p * \langle 0 \rangle}(s') & \text{if~${s = \langle 0 \rangle * s'}$,}\\
			g_{p * \langle n \rangle}(s') & \text{if~${s = \langle n \rangle * s'}$ for positive~${n}$.}
		\end{cases}
	\end{equation*}
	This family is defined simultaneously for all prefixes~${p \in T_f}$ by induction along the length of~${s}$. Under the assumption~${p * s \in T_f}$, we may conclude~${(p * \langle 0 \rangle) * s' \in T_f}$ in the second case and~${(p * \langle n \rangle) * s' \in T_f}$ in the third one. Hence, our recursive invocations are well-formed. Moreover, we clearly apply a shorter sequence~${s'}$ to our family in each case, which ensures that~${g_p(s)}$ can, in fact, be defined using recursion.
	
	We still have to verify that the values of each~${g_p}$ for~${p \in T_f}$ lie within~${2^{T_f}}$. For this, we prove by induction along the length of~${s}$ that~${g_p(s) \in 2^{T_f}}$ holds and that each exponent occurring in~${g_p(s)}$ is less than or equal to~${p}$ in~${T_f}$, for each~${p * s \in T_f}$.
	Let~${p * s \in T_f}$ be arbitrary and assume that our claim has already been shown for all sequences~${q * t \in T_f}$ with~${|t| < |s|}$. If~${s}$ is the empty sequence, our claim is trivial. Otherwise, if we are in the second case with~${s = \langle 0 \rangle * s'}$, our induction hypothesis tells us that~${g_{p * \langle 0 \rangle}(s')}$ lies in~${2^{T_f}}$ and all of its exponents are less than or equal to~${p * \langle 0 \rangle}$. Hence, they lie strictly below~${p}$, which entails that~${g_p(s) \in 2^{T_f}}$ holds. Of course, all exponents in~${g_p(s)}$ now also are less than or equal to~${p}$. The case for~${s = \langle n \rangle * s'}$ with~${n > 0}$ is even more direct.
	
	Using induction along~${|s| + |t|}$, we show that~${g_p(s) > g_p(t)}$ holds for any two sequences~${p * s, p * t \in T_f}$ with~${p * s < p * t}$ and~${|s| < |t|}$. The former implies that~${s}$ cannot be empty, while the latter implies that~${t}$ cannot be empty. Thus, let~${n, m \in \n}$ and~${s', t' \in \n^*}$ be such that~${s = \langle n \rangle * s'}$ and~${t = \langle m \rangle * t'}$ hold. If both~${n}$ and~${m}$ are positive, they must be equal by Lemma~\ref{lem:pos_equal}. Using the induction hypothesis, we conclude~${g_p(s) = g_{p * \langle n \rangle}(s') > g_{p * \langle m \rangle}(t') = t}$.
	If both~${n}$ and~${m}$ are equal to zero, we have~${g_p(s) = 2^p + g_{p * \langle 0 \rangle}(s') > 2^p + g_{p * \langle 0 \rangle}(t') = g_p(t)}$. Next, if~${n}$ is equal to zero and~${m}$ is positive, we have~${g_p(s) > g_p(t)}$ since~${g_p(s)}$ begins with~${2^p}$ while all exponents in~${g_p(t) = g_{p * \langle m \rangle}(t')}$ are less than or equal to~${p * \langle m \rangle}$ and, hence, lie strictly below~${p}$. Finally, if~${n}$ is positive and~${m}$ is equal to zero, this contradicts our assumption~${p * s < p * t}$.
	
	Finally, we are looking for a strictly ascending sequence~${(s^n)_{n \in \n} \subseteq T_f}$ such that its members have increasing length. In order to invoke Lemma~\ref{lem:T_f_perfect_sequence}, we require an infinite subset of~${T_f}$. We can simply take~${T_f}$ itself, but for this, we need to verify that it is, indeed, an infinite set: For each~${n \in \n}$, we can define a sequence~${s \in T_f}$ of length~${n}$ by setting~${s_i}$ to be the unique witness satisfying~${f(s_i) = i}$ if it exists and~${s_i < n}$ holds, for each index~${i < n}$. Otherwise, of course, we set~${s_i := 0}$. For each index~${i < n}$, the value of~${s_i}$ can be determined by a bounded search. Hence, it can be constructed using the axioms of~${\rca_0}$. Now that we have found a sequence in~${T_f}$ for each length, we conclude that~${T_f}$ is infinite. Thus, we invoke Lemma~\ref{lem:T_f_perfect_sequence}, which yields an infinite sequence~${(s^n)_{n \in \n} \subseteq T_f}$ that, in particular, satisfies~a) and~d), i.e.~its members strictly increase in length and~${(s^n)_{n \in \n}}$ is strictly ascending. Therefore,~${(g_{\langle\rangle}(s^n))_{n \in \n}}$ is infinitely descending in~${2^{T_f}}$.
\end{proof}
\begin{proof}[Proof of Theorem~\ref{thm:exp_perfect}]
	Assume that~${2^\alpha}$ is well-founded for each well order~${\alpha}$ such that each infinite suborder of~${\alpha}$ contains a strictly ascending sequence. In order to derive arithmetical comprehension, we consider an arbitrary injective function~${f: \nplus \to \n}$. If we are able to construct the range of any such~${f}$, then our claim follows by \cite[Lemma III.1.3]{SimpsonBook}.
	
	Consider~${T_f}$, which is not only a linear order by Lemma~\ref{lem:T_f_linear_order} but also, if the range of~${f}$ does not exist as a set, satisfies the property that each infinite suborder contains a strictly ascending sequence by Lemma~\ref{lem:T_f_perfect_sequence}. It is easy to see that the latter property already entails that~${T_f}$ must be well-founded: Consider, for contradiction, an infinite descending sequence in~${T_f}$. Using the usual arguments, we can assume that the codes of its members are strictly increasing. This allows us to consider the collection of all its members as an infinite subset of~${T_f}$ by~${\Delta^0_1}$-comprehension. Now, already Lemma~\ref{lem:T_f_perfect_sequence}~c) leads to a contradiction.
	
	By Lemma~\ref{lem:2_T_f_ill_founded}, we know that~${2^{T_f}}$ is ill-founded. Hence, it the range of~${f}$ does not exist as a set, then~${T_f}$ contradicts our assumption. We conclude that the range of~${f}$ is a set and, since~${f}$ was an arbitrary function, that arithmetical comprehension holds by \cite[Lemma III.1.3]{SimpsonBook}.
\end{proof}

\section{Bootstrapping arithmetical comprehension}\label{sec:boot}

In this section, we connect Theorem~\ref{thm:exp_perfect} with the uniform Kruskal theorem so that together with Proposition~\ref{prop:ukt_ipp}, it yields our proof of Theorem~\ref{thm:UKT_RCA}. Similar to Section~\ref{sec:ipp}, we require a certain class of \po-dilators for this:

\begin{definition}[{\cite[Example~2.3]{FRWUKT}}]\label{def:W_alpha}
	For each well order~${\alpha}$, we define a functor~${W_\alpha}$ that maps each partial order~${X}$ to
	\begin{equation*}
		W_\alpha(X) := 1 + (\alpha \times X)\period
	\end{equation*}
	Given an embedding~${f: X \to Y}$ between linear orders~${X}$ and~${Y}$, we define
	\begin{equation*}
		W_\alpha(f)(\sigma) :=
		\begin{cases}
			(0, 0) & \text{if~${\sigma = (0, 0)}$,}\\
			(1, (\beta, f(x))) & \text{if~${\sigma = (1, (\beta, x))}$,}
		\end{cases}
	\end{equation*}
	for all~${\sigma \in W_\alpha(X)}$.
	The support of each such~${\sigma}$ is given by
	\begin{equation*}
		\supp_X(\sigma) :=
		\begin{cases}
			\emptyset & \text{if~${\sigma = (0, 0)}$,}\\
			\{x\} & \text{if~${\sigma = (1, (\beta, x))}$.}
		\end{cases}
	\end{equation*}
	In order to ease notation, we write~${0}$ and~${(\beta, x)}$ instead of~${(0, 0)}$ and~${(1, (\beta, x))}$, respectively, for elements in~${W_\alpha(X)}$ given a well order~${\alpha}$ and a partial order~${X}$.
\end{definition}
In~${\rca_0}$, it can be seen that~${W_\alpha}$ is a \po-dilator for each well order~${\alpha}$ (cf.~\cite[Example~2.6]{FRWUKT}). Additionally, this \po-dilator is normal (cf.~\cite[Example~3.3]{FRWUKT}). In fact, both of these results also hold if~${\alpha}$ is ill-founded.

However, the system~${\rca_0}$ is not able to prove that~${W_\alpha}$ is a \wpo-dilator for arbitrary well orders~${\alpha}$: Assuming that the ascending descending sequence principle does not hold, we find an infinite well order~${\alpha}$ such that even~${\alpha'}$, which results from~${\alpha}$ if we consider its opposite order, is a well order. It is easy to see that~${W_\alpha(\alpha')}$ contains a bad sequence (cf.~\cite{Towsner}, see also \cite[Lemma~3.2]{FUKruskal}). This property of~${W_\alpha}$ is the reason why Freund, Rathjen, and Weiermann proved their equivalence using the chain antichain principle and why Freund and the author required the ascending descending sequence principle.

\begin{lemma}\label{lem:2_alpha_W_alpha}
	The system~${\rca_0}$ proves that for any well order~${\alpha}$ and Kruskal fixed point~${X}$ of~${W_\alpha}$, there exists a quasi embedding from~${2^\alpha}$ into~${X}$.
\end{lemma}
\begin{proof}
	In \cite[Example~3.9]{FRWUKT}, a quasi embedding from~${\alpha^*}$ into~${X}$ is constructed. It can easily be seen that~${2^\alpha}$ quasi embeds into~${\alpha^*}$ by mapping~${2^{\beta_0} + \dots + 2^{\beta_{n-1}}}$ in~${2^\alpha}$ to~${\langle \beta_0, \dots, \beta_{n-1} \rangle}$ in~${\alpha^*}$.
\end{proof}

\begin{lemma}\label{lem:UKT_ACA}
	The system~${\rca_0}$ proves that the uniform Kruskal theorem implies arithmetical comprehension.
\end{lemma}
\begin{proof}
	Let~${\alpha}$ be an arbitrary well order such that each infinite suborder of~${\alpha}$ contains a strictly ascending sequence. In order to derive arithmetical comprehension using Theorem~\ref{thm:exp_perfect}, we only have to show that for any such~${\alpha}$, the order~${2^\alpha}$ is well-founded as well.
	
	Consider the normal \po-dilator~${W_\alpha}$ from Definition~\ref{def:W_alpha}. We show that it is already a \wpo-dilator. Let~${X}$ be an arbitrary well partial order and assume, for contradiction, that~${(\beta_n, x_n)_{n \in \n} \subseteq \alpha \times X}$ is a bad sequence in~${W_\alpha(X)}$.\footnote{If~${0 \in W_\alpha(X)}$ occurs in our bad sequence, we can immediately remove it as this member may only appear once.}
	
	First, we use the infinite pigeonhole principle in order to ensure that we can assume the sequence~${(\beta_n)_{n \in \n}}$ to have increasing codes. For this, let~${n \in \n}$ be arbitrary and assume, for contradiction, that the codes of all~${\beta_m}$ for~${m > n}$ are less than or equal to that of~${\beta_n}$. By the infinite pigeonhole principle, which follows from the uniform Kruskal theorem using Proposition~\ref{prop:ukt_ipp}, this yields an element~${\beta' \in \alpha}$ (with a code less than or equal to that of~${\beta_n}$) that is equal to~${\beta_m}$ for infinitely many indices~${m > n}$. W.l.o.g., we can now assume that~${(\beta_n)_{n \in \n}}$ is constant. Then, however,~${(x_n)_{n \in \n}}$ must be bad in~${X}$, which contradicts our assumption that~${X}$ is a well partial order.
	
	We conclude that for each~${n \in \n}$, there exists~${m > n}$ such that the code of~${\beta_m}$ is strictly greater than that of~${\beta_n}$. W.l.o.g., we can now assume that the code of~${\beta_{n+1}}$ is strictly greater than that of~${\beta_n}$, for all~${n \in \n}$. This allows us to construct an infinite suborder~${B \subseteq \alpha}$ that consists of exactly those elements that appear in~${(\beta_n)_{n \in \n}}$. By assumption, this suborder~${B}$ contains a strictly ascending sequence. By moving to this subsequence, we can assume, w.l.o.g., that already~${(\beta_n)_{n \in \n}}$ is strictly ascending. Similar to before, however, this entails that~${X}$ must be bad. Again, this contradicts our assumption that~${X}$ is a well partial order. We conclude that~${W_\alpha}$ is a \wpo-dilator.
	
	Now that~${W_\alpha}$ is a \wpo-dilator, we can apply the uniform Kruskal theorem, which yields a well partial order~${X}$ with a collapse~${\kappa: W_\alpha(X) \to X}$ that together form a Kruskal fixed point of~${W_\alpha}$. Using the quasi embedding from Lemma~\ref{lem:2_alpha_W_alpha}, we see that~${2^\alpha}$ must be well-founded. Hence, since~${\alpha}$ was arbitrary, arithmetical comprehension holds by Theorem~\ref{thm:exp_perfect}.
\end{proof}
\begin{proof}[Proof of Theorem~\ref{thm:UKT_RCA}]
	Concerning the direction from~${\Pi^1_1}$-comprehension to the uniform Kruskal theorem, we simply apply Theorem~\ref{thm:UKT_CAC}. For the other direction, from the uniform Kruskal theorem to~${\Pi^1_1}$-comprehension, we first use Lemma~\ref{lem:UKT_ACA} to bootstrap arithmetical comprehension, which implies the chain antichain principle, and apply Theorem~\ref{thm:UKT_CAC} thereafter.
\end{proof}

\section{Weak K\H{o}nig's lemma and \po-dilators}\label{sec:wkl}

In \cite[Corollary~4.1.20]{UftringPhD}, the author proved that there is no predilator that, over~${\rca_0}$, becomes a dilator if and only if weak K\H{o}nig's lemma holds. In joint work with Anton Freund (cf.~\cite[Theorem~3.1]{FUConservative}), the employed ideas were used to generalize the famous result that~${\wkl_0}$ is conservative over~${\rca_0}$ from~${\Pi^1_1}$-statements to formulas that universally quantify over well orders (and use, like~${\Pi^1_1}$-statements, only arithmetical quantification thereafter).

In \cite[Proposition~5.3.7]{UftringPhD} (see also \cite[Theorem~3.3]{FUConservative}), the author showed, however, that it is possible to characterize weak K\H{o}nig's lemma using the statement that for all well partial orders a certain~${\Sigma^0_1}$-property is satisfied. Hence, the result that weak K\H{o}nig's lemma cannot be characterized using dilators in the way described above does not extend to~${\wpo}$-dilators using the original ideas. Still, we have the following result (cf.~\cite[Proposition~5.3.6]{UftringPhD}):
\begin{theorem}\label{thm:wpo_dilator_wkl}
	There is no~${\po}$-dilator~${W}$ such that, over~${\rca_0}$, weak K\H{o}nig's lemma holds if and only if~${W}$ is a~${\wpo}$-dilator.
\end{theorem}
In order to prove this, we proceed by a case distinction on whether the \po-dilator in question is \emph{unary} or not. Let us write~${|X|}$ for the cardinality of a set~${X}$. Then, consider the following:
\begin{definition}[{\cite[Definition~3.1]{FUKruskal}}]
	We call a~${\po}$-dilator~${W}$ \emph{unary} if and only if for each partial order~${X}$ and element~${\sigma \in W(X)}$, we have~${|\supp_X(\sigma)| \leq 1}$.
\end{definition}
\begin{lemma}\label{lem:ipp_unary}
	The system~${\rca_0}$ proves that if there is a non-unary \wpo-dilator, then the infinite pigeonhole principle holds.
\end{lemma}
A similar result has already been shown by Anton Freund and the author in \cite[Lemma~3.2]{FUKruskal} for the ascending descending sequence principle and \emph{normal} \wpo-dilators. This, already, would yield Theorem~\ref{thm:wpo_dilator_wkl} if we restricted ourselves to normal \po-dilators. As we will see, Lemma~\ref{lem:ipp_unary} even allows us to state this theorem for \emph{all}~${\po}$-dilators.

Before we continue with the proof of Lemma~\ref{lem:ipp_unary}, we require some further notion:
\begin{definition}[{\cite[Definition~1.6]{FUKruskal}}]
	A \po-dilator~${W: \po \to \po}$ is \emph{monotone} if for any two quasi embeddings~${f, g: X \to Y}$ for partial orders~${X}$ and~${Y}$ with~${f \leq g}$, i.e.~satisfying~${f(x) \leq g(x)}$ for all~${x \in X}$, we have~${W(f) \leq W(g)}$.
\end{definition}
\begin{lemma}[{\cite[Proposition~2.1]{FUKruskal}}]\label{lem:monotone}
	The system~${\rca_0}$ proves that any \wpo-dilator is monotone.
\end{lemma}
Now, everything is ready for proving our result:
\begin{proof}[Proof of Lemma~\ref{lem:ipp_unary}]
	Let~${W}$ be a non-unary \wpo-dilator witnessed by a partial order~${X}$ and an element~${\sigma \in W(X)}$ with~${|\supp_X(\sigma)| \geq 2}$. Using the canonical embedding from~${\supp_X(\sigma)}$ into~${X}$ together with the support condition, we can assume that already~${\supp_X(\sigma) = X}$ holds.
	
	We pick two distinct elements~${x, y \in X}$ and extend~${X}$ to~${X' := X \cup \{x', y'\}}$, where~${x'}$ and~${y'}$ are two distinct new elements that were not contained in~${X}$. We set~${z < x'}$ (or~${z < y'}$) if and only if~${z \leq x}$ (or~${z \leq y}$) holds, for all~${z \in X}$. It is easy to see that~${X'}$ is a partial order. We define two functions~${f_x, f_y: X \to X'}$ with
	\begin{equation*}
		f_x(z) :=
		\begin{cases}
			x' & \text{if~${z = x}$,}\\
			z & \text{otherwise,}
		\end{cases}
		\qquad
		f_y(z) :=
		\begin{cases}
			y' & \text{if~${z = y}$,}\\
			z & \text{otherwise,}
		\end{cases}
	\end{equation*}
	for all~${z \in X}$. It is easy to see that these functions are (quasi) embeddings and, hence, we can define elements~${\sigma_x := W(f_x)(\sigma)}$ and~${\sigma_y := W(f_y)(\sigma)}$ in~${W(X')}$. Since their supports are different, we conclude that~${\sigma_x}$ and~${\sigma_y}$ are distinct elements. W.l.o.g., we have~${\sigma_x \nleq \sigma_y}$.
	
	Assume, for contradiction, that the infinite pigeonhole principle does not hold. By Lemma~\ref{lem:ipp_alpha_times_n}, there is a well order~${\alpha}$ together with a natural number~${n \in \n}$ such that~${\alpha \times n}$ contains a bad sequence~${((\beta_i, n_i))_{i \in \n}}$. Let~${N}$ be the finite antichain of size~${n}$. Clearly, the same bad sequence also lives in~${\alpha \times N}$. W.l.o.g., we can assume~${\alpha}$ to contain a smallest element~${\bot}$ that does not occur in this sequence.
	
	Let~${Y := N + \alpha + X \setminus \{x, y\}}$. Either using the fact that~${N}$ and~${X}$ are finite or using the infinite pigeonhole principle for three colors, it can be seen that~${Y}$ is a well partial order.
	For each number~${i \in \n}$, we define a function~${g_i: X \to Y}$ with
	\begin{equation*}
		g_i(z) :=
		\begin{cases}
			(0, n_i) & \text{if~${z = x}$,}\\
			(1, \beta_i) & \text{if~${z = y}$,}\\
			(2, z) & \text{otherwise,}
		\end{cases}
	\end{equation*}
	for all~${z \in X}$. It is not hard to see that~${g_i}$ is a quasi embedding for each~${i \in \n}$. Using the assumption that~${W}$ is a \wpo-dilator, we conclude that~${W(Y)}$ must be a well partial order. Hence, there are indices~${i < j}$ such that~${W(g_i)(\sigma) \leq W(g_j)(\sigma)}$ holds. By assumption, we have~${(\beta_i, n_i) \nleq (\beta_j, n_j)}$, which implies one of~${n_i \neq n_j}$ or both~${n_i = n_j}$ and~${\beta_i \nleq \beta_j}$.
	
	In the latter case, we have~${\beta_i > \beta_j}$ since~${\alpha}$ is a well order. By definition of the family~${(g_i)_{i \in \n}}$, this leads to~${g_i \geq g_j}$ and, because~${W}$ is monotone as a \wpo-dilator by Lemma~\ref{lem:monotone}, we arrive at the contradiction~${W(g_i)(\sigma) > W(g_j)(\sigma)}$. Notice that this inequality is strict since both elements have distinct supports.
	
	In the former case with~${n_i \neq n_j}$, we consider the function~${g'_i: X \to Y}$ with
	\begin{equation*}
		g'_i(z) :=
		\begin{cases}
			(0, n_i) & \text{if~${z = x}$,}\\
			(1, \bot) & \text{if~${z = y}$,}\\
			(2, z) & \text{otherwise,}
		\end{cases}
	\end{equation*}
	for all~${z \in X}$. Notice that it is defined like~${g_i}$ but maps~${y}$ to~${(1, \bot)}$ instead of~${(1, \beta_i)}$. Using a similar argument as before, we see that~${g'_i}$ is a quasi embedding.
	In order to arrive at our contradiction, we define a function~${h: X' \to Y}$ with
	\begin{equation*}
		h(z) :=
		\begin{cases}
			(0, n_j) & \text{if~${z = x}$,}\\
			(0, n_i) & \text{if~${z = x'}$,}\\
			(1, \bot) & \text{if~${z = y}$,}\\
			(1, \beta_j) & \text{if~${z = y'}$,}\\
			(2, z) & \text{otherwise,}
		\end{cases}
	\end{equation*}
	for all~${z \in X'}$. Using the facts~${n_i \neq n_j}$ and~${\bot < \beta_j}$, we derive that~${h}$ is a quasi embedding. It is easy to see that~${g'_i = h \circ f_x}$ and~${g_j = h \circ f_y}$ hold. Since~${W}$ is monotone, our assumption~${W(g_i)(\sigma) \leq W(g_j)(\sigma)}$ leads to
	\begin{equation*}
		W(h \circ f_x)(\sigma) = W(g'_i)(\sigma) \leq W(g_i)(\sigma) \leq W(g_j)(\sigma) = W(h \circ f_y)(\sigma)\period
	\end{equation*}
	Now,~${W(h)}$ is a quasi embedding, which entails
	\begin{equation*}
		\sigma_x = W(f_x)(\sigma) \leq W(f_y)(\sigma) = \sigma_y\period
	\end{equation*}
	Recall that we made the assumption~${\sigma_x \nleq \sigma_y}$.
	From this contradiction, we conclude that the infinite pigeonhole principle must hold.
\end{proof}
As a final ingredient for our proof of Theorem~\ref{thm:wpo_dilator_wkl}, we show that a particular class of \po-dilators already consists of \wpo-dilators under a weak assumption:
\begin{lemma}\label{lem:cac_unary_monotone}
	The system~${\rca_0}$ together with the chain antichain principle proves any unary and monotone \po-dilator~${W}$ to be a \wpo-dilator if~${W(1)}$ is a well partial order.
\end{lemma}
Compare this to \cite[Lemma~3.3]{FUKruskal}, where we showed that this class of \po-dilators has Kruskal fixed points that are well partial orders if and only if arithmetical comprehension is available.
\begin{proof}[Proof of Lemma~\ref{lem:cac_unary_monotone}]
	Using the canonical embedding from~${0}$ into~${1}$, we see that~${W(0)}$ must also be a well partial order. By the chain antichain principle, which implies that products of well partial orders are, again, well partial orders, we know that~${(W(0) + W(1)) \times X}$ contains no bad sequences for any well partial order~${X}$. 
	We have already seen that~${W(X)}$ is a well partial order if~${X}$ is empty. Now, if~${X}$ is a nonempty well partial order (witnessed by some element~${z \in X}$), then we derive our claim by constructing a quasi embedding~${f: W(X) \to (W(0) + W(1)) \times X}$.
	Let~${\iota: 0 \to X}$ denote the unique embedding with this signature.
	We define
	\begin{equation*}
		f(\sigma) :=
		\begin{cases}
			((0, \sigma_0), z) & \text{if~${\sigma = W(\iota)(\sigma_0)}$,}\\
			((1, \sigma_0), x) & \parbox[t]{19em}{if~${\sigma = W(g_x)(\sigma_0)}$ where~${g_x: 1 \to X}$ maps~${0}$\\to the unique element in~${\supp_X(\sigma)}$,}
		\end{cases}
	\end{equation*}
	for all~${x \in X}$. Since our \po-dilator~${W}$ is unary, each~${\sigma \in W(X)}$ either has an empty support and lies in the range of~${W(\iota)}$ or it has exactly one element~${x \in X}$ in its support and lies in the range of~${W(g_x)}$. Hence, all cases are covered. Moreover, since~${W(\iota)}$ and~${W(g_x)}$ are quasi embeddings,~${\sigma_0}$ is unique in each case.
	
	We show that~${f}$ is a quasi embedding: Consider two elements~${\sigma, \tau \in W(X)}$ with~${f(\sigma) \leq f(\tau)}$. First, by definition of~${f}$, we know that either both elements have an empty support or both of them have a support containing exactly one element each. In the first case, consider~${\sigma_0, \tau_0 \in W(0)}$ with~${\sigma = W(\iota)(\sigma_0)}$ and~${\tau = W(\iota)(\tau_0)}$. By assumption, we have~${\sigma_0 \leq \tau_0}$, which leads to~${\sigma \leq \tau}$ since~${W(\iota)}$ is an embedding. (Recall that \po-dilators preserve embeddings.) In the other case, consider~${x, y \in X}$ with~${\{x\} = \supp_X(\sigma)}$ and~${\{y\} = \supp_X(\tau)}$ as well as~${\sigma_0, \tau_0 \in W(1)}$ with~${\sigma = W(g_x)(\sigma_0)}$ and~${\tau = W(g_y)(\tau_0)}$. By assumption, we have both~${\sigma_0 \leq \tau_0}$ and~${x \leq y}$. Similar to before, the former yields~${W(g_x)(\sigma_0) \leq W(g_x)(\tau_0)}$. By the fact that~${W}$ is monotone, we also have~${W(g_x)(\tau_0) \leq W(g_y)(\tau_0)}$. In conclusion, we arrive at~${\sigma \leq \tau}$.
\end{proof}
\begin{proof}[Proof of Theorem~\ref{thm:wpo_dilator_wkl}]
	Assume, for contradiction, that there exists a~${\po}$-dilator~${W}$ with the claimed property.
	
	We proceed by case distinction on the property of being a unary \po-dilator:
	Assume that it is true (in the standard model) that there exists a partial order~${X}$ together with an element~${\sigma \in W(X)}$ satisfying~${|\supp_X(\sigma)| \geq 2}$. By the support condition, we may assume~${X}$ to be a finite partial order. Now, using the way that~${\po}$-dilators are coded together with the fact that~${X}$ is finite, this statement can be formulated as a~${\Sigma^0_1}$-sentence. Hence, it is provable in~${\rca_0}$. Now, using Lemma~\ref{lem:ipp_unary}, we can argue in~${\rca_0}$ as follows: If weak K\H{o}nig's lemma holds, then~${W}$ is (by assumption) a non-unary~${\wpo}$-dilator, which yields the infinite pigeonhole principle. However, there is a model of~${\wkl_0}$, in which the infinite pigeonhole principle is \emph{not} satisfied (since~${\wkl_0}$ is~${\Pi^1_1}$-conservative over~${\rca_0}$, which does not prove the infinite pigeonhole principle, see~\cite[Corollary~IX.2.6]{SimpsonBook} and \cite[Proposition~7]{ParisKirby}).
	
	This contradiction entails that in the standard model, there cannot exist a partial order~${X}$ and an element~${\sigma \in W(X)}$ with~${|\supp_X(\sigma)| \geq 2}$. Thus, it is true that~${W}$ is a unary \wpo-dilator. In particular,~${W}$ is a unary and monotone \po-dilator (see Lemma~\ref{lem:monotone}) such that~${W(1)}$ is a well partial order. Since these properties are~${\Pi^1_1}$ (see the discussion below \cite[Corollary~2.6]{FUKruskal}), they are satisfied in any~${\omega}$-model of~${\rca_0}$. In particular, this holds for an~${\omega}$-model of the chain antichain principle that rejects weak K\H{on}ig's lemma. The existence of such a model was shown by Hirschfeldt and Shore (cf.~\cite[Corollary~3.11]{HS}). However, in the presence of the chain antichain principle, Lemma~\ref{lem:cac_unary_monotone} tells us that our unary and monotone \po-dilator~${W}$, for which~${W(1)}$ is a well partial order, must automatically be a \wpo-dilator. This contradicts the assumption that weak K\H{o}nig's lemma does not hold in our considered~${\omega}$-model.
\end{proof}
\bibliographystyle{amsplain}
\bibliography{ukt}
\end{document}